\documentclass[11pt, a4paper]{amsart}

\usepackage{amsfonts,amsmath,amssymb, amscd}

\newtheorem{Lem}{Lemma}[section]
\newtheorem{Prop}[Lem]{Proposition}
\newtheorem{Cor}[Lem]{Corollary}
\newtheorem{Thm}[Lem]{Theorem}

\theoremstyle{definition}

\newtheorem{Rem}[Lem]{Remark}
\newtheorem{Qu}[Lem]{Question}

\newcommand\pf{\begin{proof}}
\newcommand\epf{\end{proof}}

\newcommand\Pal{\operatorname{Pal}}
\newcommand\Aut{\operatorname{Aut}}

\newcommand\Hom{\operatorname{Hom}}

\newcommand\Ker{\operatorname{Ker}}

\newcommand\id{\operatorname{id}}
\newcommand\GL{\operatorname{GL}}
\newcommand\SL{\operatorname{SL}}

\newcommand\PP{\operatorname{P}}
\newcommand\RR{\mathcal R}

\newcommand\ZZ{\mathbb Z}

\numberwithin{equation}{section}

\title{A palindromization map for the free group}

\author{Christian Kassel}
\address{Christian Kassel: 
Universit\'{e} de Strasbourg,
Institut de Recherche Math\'{e}\-ma\-tique Avanc\'{e}e,
CNRS - Universit\'{e} Louis Pasteur,
7 rue Ren\'{e} Descartes, 67084 Strasbourg, France}
\email{kassel@math.u-strasbg.fr}

\author{Christophe Reutenauer}
\address{Christophe Reutenauer:
Math\'ematiques, Universit\'e du Qu\'ebec \`a Montr\'eal,
Montr\'eal, CP 8888, succ.\ Centre Ville, Canada H3C 3P8}
\email{reutenauer.christophe@uqam.ca}

\begin{document}

\maketitle

\vskip 30pt
\noindent
{\sc Abstract}. 
\emph{ 
We define a self-map $\Pal: F_2 \to F_2$ of the free group
on two generators $a, b$,
using automorphisms of~$F_2$ that form a group isomorphic to
the braid group~$B_3$.
The map~$\Pal$ restricts to de Luca's right iterated palindromic closure
on the submonoid generated by $a, b$. We show that~$\Pal$
is continuous for the profinite topology on~$F_2$;
it is the unique continuous extension of 
de Luca's right iterated palindromic closure to~$F_2$.
The values of~$\Pal$ are palindromes and coincide with
the elements $g \in F_2$ such that $abg$ is conjugate to $bag$.
}



\bigskip\bigskip

\hspace{3cm}

\section*{Introduction}

To any word~$w$ on an alphabet consisting of two letters $a$ and~$b$, 
de Luca~\cite{Lu} associated a palindromic word~$\PP(w)$, 
called its \emph{right iterated palindromic closure}.
The element $\PP(w) \in \{a,b\}^*$ is defined recursively 
by $\PP(1) \nolinebreak = \nolinebreak 1$ and 
$$\PP(wx) = \bigl(\PP(w)x \bigr)^+$$
for all $w \in \{a,b\}^*$ and~$x\in \{a,b\}$;
here $w^+$ is the unique shortest palindrome having $w$ as a prefix.
De~Luca showed that all words $\PP(w)$ are \emph{central} in the sense of~\cite{Lo},
and that any central word is of this form.
Moreover, de Luca's map $\PP$ is injective, i.e.,
$\PP(u) = \PP(v)$ implies $u = v$. 

In this paper we construct a self-map
$$\Pal : F_2 \to F_2$$
of the free group $F_2$ on~$a,b$,
whose restriction to the monoid~$\{a,b\}^*$ is de~Luca's map
$\PP : \{a,b\}^* \to \{a,b\}^*$,
i.e., we have $\Pal(w) = \PP(w)$ for all $w\in \{a,b\}^*$.
The map~$\Pal$, which we call the \emph{palindromization map},
is defined using certain automorphisms of~$F_2$.
These automorphisms form a group that is isomorphic to
the \emph{braid group}~$B_3$ of braids on three strands.
One of the most interesting properties of our map~$\Pal$ 
is that it is continuous for the \emph{profinite topology} on~$F_2$;
we actually prove that $\Pal$ is the unique continuous extension of~$\PP$
to~$F_2$.
As is the case with the map~$\PP$, each image $\Pal(w)$
of the palindromization map is a palindrome, i.e., is fixed by the unique
anti-automorphism of~$F_2$ fixing $a$ and~$b$.
We also characterize the elements $g \in F_2$ belonging to the image of~$\Pal$
as those for which $abg$ and $bag$ are conjugate elements.
Contrary to~$\PP$, our map~$\Pal$ is not injective; 
we determine all pairs $(u,v)$ such that $\Pal(u) = \Pal(v)$;
the result is expressed in terms of the braid group~$B_3$.

To shed light on the theory of Sturmian words and morphisms, 
it is convenient to put it into the context of the free group.
In this paper we illustrate this idea on de Luca's right iterated palindromic closure.
It was precisely by connecting the combinatorics of words and 
group theory that we were able to find our extension~$\Pal$.
Indeed, Justin~\cite{Ju} proved that de Luca's map~$\PP$ 
satisfies a certain functional equation.
We interpret this as expressing that~$\PP$ is a \emph{cocycle} in the sense
of Serre's non-abelian cohomology. 
In this language our main observation is that~$\PP$ is a \emph{trivial} cocycle. 
This has two consequences, of which the second one is quite fortunate: 
firstly, to express the triviality of the cocycle~$\PP$ 
we are forced to work in the free group~$F_2$;
secondly, expressing~$\PP$ as a trivial cocycle yields \textit{ipso facto} 
a formula for~$\Pal$.

Let us detail the contents of the paper.
In Section~\ref{automorphisms}
we associate an automorphism~$R_w$ of~$F_2$ to each $w\in F_2$;
the automorphisms~$R_w$ are exactly the automorphisms of~$F_2$ fixing~$aba^{-1}b^{-1}$.
In Section~\ref{braid} we show that the automorphisms~$R_w$
form a subgroup that is isomorphic to the braid group~$B_3$.
In Section~\ref{def-pal} we define the palindromization map $\Pal : F_2 \to F_2$
and we give it a cohomological interpretation.
In Section~\ref{pal-pal} we show that each element $\Pal(w)$ is a palindrome in~$F_2$
and that our map~$\Pal$ extends de Luca's right iterated palindromic closure;
we also compute the image of~$\Pal(w)$ in the free abelian group~$\ZZ^2$.
As mentioned above, the map~$\Pal$ is not injective; in Section~\ref{triv-pal}
we determine all pairs $(u,v)$ such that $\Pal(u) = \Pal(v)$.
In Section~\ref{profinite} we establish that $\Pal$ is continuous
for the profinite topology on~$F_2$.
We characterize the elements of~$F_2$
belonging to the image of~$\Pal$ in Section~\ref{sect-image}.
The paper concludes with a short appendix
collecting the basic facts on non-abelian cohomology
needed in Section~\ref{def-pal}.

\section{The automorphisms~$R_w$}\label{automorphisms}

Let $F_2$ be the free group generated by $a$ and~$b$.
Consider the automorphisms $R_a$, $R_b$ of~$F_2$ defined by
\begin{equation}\label{def-R}
R_a   = 
\begin{cases}
a \mapsto a \, , \\
b \mapsto b a \, ,
\end{cases} 
\quad\text{and}\quad
R_b =
\begin{cases}
a \mapsto ab \, , \\
b \mapsto b \, .
\end{cases}
\end{equation}
Their inverses are given by
\begin{equation*}
(R_a)^{-1}   = 
\begin{cases}
a \mapsto a \, , \\
b \mapsto ba^{-1} \, ,
\end{cases} 
\quad\text{and}\quad
(R_b)^{-1} =
\begin{cases}
a \mapsto ab^{-1} \, , \\
b \mapsto b \, .
\end{cases}
\end{equation*}
The automorphisms $R_a$ and $R_b$ are respectively denoted by $\widetilde{G}$, $\widetilde{D}$
in~\cite[Sect.\ 2.2.2]{Lo} (see also~\cite[Eqn.~(2.1)]{KR}).
They are related by
\begin{equation}\label{E}
E \, R_a = R_b  \, E \, ,
\end{equation}
where $E$ is the involution of~$F_2$ exchanging $a$ and~$b$.
(The automorphisms $R_a, R_b$ are instances of what Godelle calls transvections in~\cite{Go}.)

Let $w \mapsto R_w$ be the group homomorphism $F_2 \to \Aut(F_2)$ sending $a$ to~$R_a$
and $b$ to~$R_b$. 
In particular, if $w=1$ is the neutral element of~$F_2$, then $R_w = \id$ (the identity of~$F_2$).
Moreover, $R_{a^{-1}} = (R_a)^{-1}$ and $R_{b^{-1}} = (R_b)^{-1}$. 
It follows from~\eqref{E} that for all $w\in F_2$,
\begin{equation}\label{Ebis}
E \, R_w = R_{E(w)} \,  E \, .
\end{equation}

\begin{Lem}\label{commutator}
Each automorphism $R_w$ fixes the commutator $aba^{-1}b^{-1}$.
\end{Lem}

\pf
It is enough to verify that both $R_a$ and $R_b$ fix~$aba^{-1}b^{-1}$.
\epf

When $w = (ab^{-1}a)^i$ for $i= 1, 2, 4$, then an easy computation shows 
that $R_w$ is given by
\begin{equation}\label{R-center1}
R_{ab^{-1}a}  = 
\begin{cases}
a \mapsto b^{-1} \, , \\
b \mapsto b a b^{-1} \, ,
\end{cases}
\end{equation}
\begin{equation}\label{R-center2}
R_{(ab^{-1}a)^2}  = 
\begin{cases}
a \mapsto ba^{-1}b^{-1} \, , \\
b \mapsto b a b^{-1}a^{-1}b^{-1} \, ,
\end{cases}
\end{equation}
\begin{equation}\label{R-center3}
R_{(ab^{-1}a)^4}  = 
\begin{cases}
a \mapsto (b a b^{-1}) \, a \, (ba^{-1}b^{-1}) \, , \\
b \mapsto (b a b^{-1}a^{-1}) \, b \, (a ba^{-1}b^{-1})  \, .
\end{cases}
\end{equation}

We can rephrase \eqref{R-center1}--\eqref{R-center3} as follows.
Let $\tau$ be the automorphism of~$F_2$ sending $a$ to~$b^{-1}$ and~$b$ to~$a$;
we have $\tau^2(a) = a^{-1}$ and $\tau^2(b) = b^{-1}$, so that $\tau^4 = \id$. 
Then for all $u\in F_2$,
\begin{equation}\label{R-center10}
R_{ab^{-1}a}(u)  = b \, \tau(u) \, b^{-1} \, ,
\end{equation}
\begin{equation}\label{R-center20}
R_{(ab^{-1}a)^2}(u)  = ba \, \tau^2(u) \, (ba)^{-1} \, ,
\end{equation}
\begin{equation}\label{R-center30}
R_{(ab^{-1}a)^4}(u)  = 
(b a b^{-1}a^{-1}) \, u \, (b a b^{-1}a^{-1})^{-1}  \, .
\end{equation}
In other words, 
$R_{(ab^{-1}a)^4}$ is  the conjugation by the commutator~$b a b^{-1}a^{-1}$.

We next consider the abelianizations of the automorphisms~$R_w$.
Let $\pi : F_2 \to \ZZ^2$ be the canonical surjection
sending the generators $a$ and~$b$ of~$F_2$
to the respective column-vectors 
$$\begin{pmatrix}
1 \\
0
\end{pmatrix} 
\quad
\text{and}
\quad
\begin{pmatrix}
0 \\
1
\end{pmatrix} 
\in \ZZ^2 \, .
$$ 
In the sequel we identify $\Aut(\ZZ^2)$ with the matrix group~$\GL_2(\ZZ)$.

For any $w\in F_2$, let $M_w$ 
be the abelianization of the automorphism~$R_w$, 
i.e., the unique automorphism~$M_w$ of~$\ZZ^2$ such that
$\pi \circ R_w = M_w \circ \pi$.
Since $w \mapsto R_w$ is a group homomorphism, so is the map $w \mapsto M_w$.
The latter is determined by its values $M_a$ and~$M_b$.
It follows from~\eqref{def-R} that $M_a$ and~$M_b$ can be identified with the matrices
\begin{equation*}
M_a = 
\begin{pmatrix}
1 &  1 \\
0 & 1
\end{pmatrix}
\quad
\text{and}
\quad
M_b = 
\begin{pmatrix}
1 &  0 \\
1 & 1
\end{pmatrix} 
\, .
\end{equation*}
Since $M_a$ and $M_b$ are of determinant one, 
we conclude that $M_w \in \SL_2(\ZZ)$ for 
all $w\in F_2$.
Formulas~\eqref{R-center1}--\eqref{R-center3} imply that
\begin{equation}\label{M-relations}
M_{a b^{-1}a}  = 
\begin{pmatrix}
0 &  1 \\
-1 & 0
\end{pmatrix}  \, , \quad
M_{(a b^{-1}a)^2}  = -I_2\, ,
\quad
M_{(a b^{-1}a)^4} = I_2\, ,
\end{equation}
where $I_2$ denotes the unit $2 \times 2$ matrix.

\section{The braid group $B_3$}\label{braid}

Let $\RR$ be the subgroup of~$\Aut(F_2)$ consisting of all automorphisms~$R_w$,
where $w\in F_2$. 
By definition, the subgroup $\RR$~is generated by $R_a$ and~$R_b$.
It follows from Lemma~\ref{commutator}
that every element of~$\RR$ fixes the commutator $aba^{-1}b^{-1}$.
The converse also holds.

\begin{Prop}\label{char-commutator}
The group $\RR$ is the subgroup of~$\Aut(F_2)$ of all automorphisms 
fixing~$aba^{-1}b^{-1}$.
\end{Prop}

\pf
Let $\varphi$ be an automorphism fixing~$aba^{-1}b^{-1}$.
Consider the anti-automorphism $\omega$ of~$F_2$ 
such that $\omega(a) = a$ and~$\omega(b) = b$. 
For any word $w \in F_2$, the word $\omega(w)$ is its mirror image.
In particular, $\omega(aba^{-1}b^{-1}) = b^{-1}a^{-1}ba$.
We have
$$(\omega \circ \varphi \circ \omega)(b^{-1}a^{-1}ba) 
= \omega \bigl( \varphi(aba^{-1}b^{-1}) \bigr) 
= \omega(aba^{-1}b^{-1}) = b^{-1}a^{-1}ba\, .$$
Now it is well known (see~\cite[Sect.~3]{Co} or~\cite{He}) that 
the subgroup of~$\Aut(F_2)$ fixing $b^{-1}a^{-1}ba$
is generated by the two automorphisms
\begin{equation}\label{values-L}
L_a   = 
\begin{cases}
a \mapsto a \, , \\
b \mapsto ab \, ,
\end{cases} 
\quad\text{and}\quad
L_b =
\begin{cases}
a \mapsto ba \, , \\
b \mapsto b \, .
\end{cases}
\end{equation}
(The automorphisms $L_a$ and $L_b$ coincide respectively with the automorphisms 
$G$ and $D$ of~\cite[Sect.~2.2.2]{Lo}; see also~\cite[Eqn.~(2.1)]{KR}.)
Therefore, $\varphi$~belongs to the subgroup of~$\Aut(F_2)$ generated by
$\omega \circ L_a \circ \omega$ and $\omega \circ L_b \circ \omega$.
A simple computation based on~\eqref{def-R} and \eqref{values-L} shows that
\begin{equation}\label{def-L}
R_a = \omega \circ L_a \circ \omega
\quad\text{and}\quad 
R_b = \omega \circ L_b \circ \omega \, .
\end{equation} 
This proves that $\varphi$ belongs to~$\RR$.
\epf

We next claim that
the group~$\RR$ is isomorphic to the braid group~$B_3$ 
of braids on three strands. 
Recall that the group~$B_3$ has the following presentation (see~\cite[Sect.~1.1]{KT}):
\begin{equation}\label{presentationB3}
B_3 = \langle \, \sigma_1, \sigma_2 \; |\;  
\sigma_1 \sigma_2 \sigma_1 = \sigma_2 \sigma_1 \sigma_2 \, \rangle \, .
\end{equation}

\begin{Prop}\label{B3}
There is a group homomorphism
$i : B_3 \to \Aut(F_2)$ such that
$i(\sigma_1) = R_a$ and $i(\sigma_2) = R_{b}^{-1}$.
This homomorphism is injective and its image is~$\RR$.
\end{Prop}

\pf
The existence of the homomorphism $i : B_3 \to \Aut(F_2)$ 
results from the relation
$$R_a \, R_b^{-1} \, R_a = R_b^{-1} \, R_a \, R_b^{-1} \, ,$$
which was observed in~\cite[Lemma~2.1]{KR}.
The image of~$i$ is clearly the subgroup~$\RR$ of~$\Aut(F_2)$
generated by~$R_a$ and~$R_b$.

In order to prove that $i$ is injective, we use further results of~\cite[Sect.~2]{KR}.
Let $B_4$ be the braid group on four strands; it has a presentation with generators
$\sigma_1$, $\sigma_2$, $\sigma_3$ and relations
$$\sigma_1 \sigma_2 \sigma_1 = \sigma_2 \sigma_1 \sigma_2\, , \quad
\sigma_2 \sigma_3 \sigma_2 = \sigma_3 \sigma_2 \sigma_3\, , \quad
\sigma_1 \sigma_3 = \sigma_3 \sigma_1 \, .$$
First observe that the natural homomorphism $j: B_3 \to B_4$ determined by  
$j(\sigma_1) = \sigma_1$ and $j(\sigma_2) = \sigma_2$ is injective.
Indeed, there is an homomorphism $q: B_4 \to B_3$ such that
$q\circ j = \id$; it is given by
$q(\sigma_1) = q(\sigma_3) = \sigma_1$ and $q(\sigma_2) = \sigma_2$;
see also~\cite[Cor.~1.14]{KT}.

In \cite[Lemma~2.6]{KR} we constructed a 
homomorphism $f: B_4 \to \Aut(F_2)$
whose kernel is the center of~$B_4$ and such that 
$$i(\sigma_1) = R_a = \tau^2 \, f(\sigma_1) \, \tau^{-2}
\quad\text{and}\quad
i(\sigma_2) = R_b^{-1} = \tau^2 \, f(\sigma_2) \, \tau^{-2}\, ,$$
where $\tau^2$ is the square of the automorphism~$\tau$ of~$F_2$ 
introduced in Section~\ref{automorphisms}.
Hence, $i(\beta) = \tau^2 \, f(\beta) \, \tau^{-2}$ for all $\beta \in B_3$.
Pick $\beta \in B_3$ in the kernel of~$i$. 
It follows from the previous relation that $f(\beta) = 1$,
which implies that $\beta$ belongs to the center of~$B_4$.
The group~$B_3$ being embedded in~$B_4$, we conclude that
$\beta$ belongs to the center of~$B_3$. Now the center of~$B_3$ is generated
by~$(\sigma_1 \sigma_2 \sigma_1)^2$. 
Therefore, $\beta = (\sigma_1 \sigma_2 \sigma_1)^{2k}$ for some~$k\in \ZZ$
and hence $i(\beta) = R_{(ab^{-1}a)^{2k}}$.
To complete the proof, it suffices to check that $R_{(ab^{-1}a)^{2k}}$ is the identity
only if~$k=0$. Consider the abelianization $M_{(ab^{-1}a)^{2k}}$ of~$R_{(ab^{-1}a)^{2k}}$.
Since $M_{(ab^{-1}a)^2}= - I_2$ by~\eqref{M-relations}, we have
$M_{(ab^{-1}a)^{2k}} = (-1)^k I_2$. The latter being the identity, $k$ must be even.
Set $k = 2 \ell$ for some integer~$\ell$.
It follows from~\eqref{R-center30}
that $R_{(ab^{-1}a)^{2k}} = R_{((ab^{-1}a)^{4})^{\ell}}$ is the conjugation
by $(b a b^{-1}a^{-1})^{\ell}$ in~$F_2$. 
Such a conjugation is the identity only if~$\ell = 0$;
therefore,~$k=0$, which implies that $\beta= 1$.
\epf 

As a consequence of Proposition~\ref{B3} we obtain the following result,
which had been observed in~\cite[Remark~2.14\,(c)]{KR}.

\begin{Cor}\label{B3bis}
There is a group isomorphism $\RR \cong B_3$.
\end{Cor}

\section{The palindromization map and a cohomological interpretation}\label{def-pal}

We use the automorphisms $R_w$ of Section~\ref{automorphisms} to define
for each $w\in F_2$ an element~$\Pal(w) \in F_2$ by the formula
\begin{equation}\label{def-Pal}
\Pal(w) = b^{-1} a^{-1} \, R_w(ab)\, .
\end{equation}
Formula~\eqref{def-Pal} defines a map 
$\Pal: F_2 \to F_2\, ;\, w \mapsto \Pal(w)$,
which we call the \emph{palindromization map}.

As a consequence of the definition and of~\eqref{def-R}--\eqref{R-center1}, 
we easily obtain the following values of~$\Pal$:
\begin{equation}\label{values0}
\Pal(w) = w \quad \text{if} \;\; w = a^{r} \ \text{or}\ b^{r} \  \text{for some} \ r \in \ZZ\, ,
\end{equation}
\begin{equation*}
\Pal(ab) = aba \, , \quad
\Pal(ba) = bab \, ,
\end{equation*}
\begin{equation*}
\Pal(ba^{-1}) = a^{-1} \, , \quad
\Pal(ab^{-1}) = b^{-1} \, ,
\end{equation*}
\begin{equation*}
\Pal(a^{-1} b) = a^{-1} b a^{-1} \, ,\quad 
\Pal(b^{-1} a) = b^{-1} a b^{-1} \, ,
\end{equation*}
\begin{equation*}
\Pal(aba^{-1}) = \Pal(ab^{-1}a^{-1})
= \Pal(bab^{-1}) = \Pal(ba^{-1}b^{-1}) = 1\, ,
\end{equation*}
\begin{equation}\label{values3}
\begin{aligned}
\Pal(ab^{-1}a) & = \Pal(b^{-1}ab^{-1}) = b^{-2} \, , \\
\Pal(ba^{-1}b) & = \Pal(a^{-1}ba^{-1}) = a^{-2}\, ,
\end{aligned}
\end{equation} 
Thus the map $\Pal$ is not injective.
In Section~\ref{triv-pal}
we will characterize the pairs $(u,v)$ of elements of~$F_2$ such that 
$\Pal(u) \nolinebreak =  \nolinebreak \Pal(v)$. 

We now prove a few properties of the palindromization map.
First, we give another formula defining it and show that $\Pal$
is invariant under the exchange automorphism~$E$.

\begin{Lem}\label{def2-pal}
For all $w\in F_2$,
\begin{equation}\label{def2-Pal}
\Pal(w) = a^{-1} b^{-1} \, R_w(ba)
\end{equation}
and
\begin{equation*}
E\bigl(\Pal(w)\bigr) = \Pal\bigl( E(w) \bigr) \, .
\end{equation*}
\end{Lem}

\pf
Since by Lemma~\ref{commutator}, $R_w$ fixes $aba^{-1}b^{-1}$, we have
\begin{eqnarray*}
aba^{-1}b^{-1} R_w(ba) 
& = & R_w(aba^{-1}b^{-1}) \, R_w (ba)
= R_w(aba^{-1}b^{-1} ba) \\
& = & R_w (ab) = ab \Pal(w)\, .
\end{eqnarray*}
Canceling on the left by~$ab$, we obtain~\eqref{def2-Pal}.

Using~\eqref{Ebis}, \eqref{def-Pal}, and~\eqref{def2-Pal},
we obtain
\begin{eqnarray*}
E\bigl(\Pal(w)\bigr) 
& = &  E \bigl(b^{-1} a^{-1} R_w(ab) \bigr)
= a^{-1}b^{-1} \, E \bigl( R_w(ab) \bigr) \\
& = & a^{-1}b^{-1} \, R_{E(w)}\bigl( E(ab) \bigr) 
=  a^{-1}b^{-1} \, R_{E(w)}(ba) \\
& = & \Pal\bigl( E(w) \bigr) \, .
\end{eqnarray*}
\epf

\begin{Cor}\label{trivialPal}
(a) For $w\in F_2$,
\begin{equation*}
\Pal(w) = 1 \;\; \Longleftrightarrow \;\; R_w(ab) = ab \; \;
\Longleftrightarrow \;\; R_w(ba) = ba \, .
\end{equation*}

(b) The set $\Pal^{-1}(1) = \{ w\in F_2 \, | \, \Pal(w) = 1\}$
is a subgroup of~$F_2$.
\end{Cor}

\pf
Part\,(a) is an immediate consequence of~\eqref{def-Pal} and~\eqref{def2-Pal}.
Part\,(b) follows from Part\,(a).
\epf

The palindromization satisfies the following important functional equation.

\begin{Prop}\label{Pal-Justin}
For all $u$, $v \in F_2$, we have
\begin{equation}\label{Justin}
\Pal(uv) = \Pal(u) \, R_u\bigl( \Pal(v) \bigr) \, .
\end{equation}
\end{Prop}

\pf
We have
\begin{eqnarray*}
ab \Pal(uv) & = & R_{uv} (ab) = R_u\bigl( R_v(ab)\bigr) \\
& = & R_u(ab \Pal(v)) = R_u(ab) \, R_u(\Pal(v)) \\
& = & ab \Pal(u) \, R_u(\Pal(v))\, .
\end{eqnarray*}
We conclude by canceling on the left by~$ab$.
\epf

Equation~\eqref{Justin} has an interesting interpretation in the language of 
Serre's \emph{non-abelian cohomology}, 
whose definition we recall in Appendix~\ref{Serre}.
The free group~$F_2$ acts on itself via the group homomorphism 
$$F_2 \to \Aut(F_2) \, ; \ w \mapsto R_w \, .$$
It follows from~\eqref{def-Pal} that the function $\Pal: F_2 \to F_2$ is a 
\emph{trivial cocycle} in the sense of Appendix~A; see~\eqref{null} with $X = ab$. 
Therefore $\Pal$ satisfies a cocycle condition of the form~\eqref{cocycle},
which in our case is nothing else that Equation~\eqref{Justin}.

Let $F_2 \rtimes F_2$ be the semi-direct product associated to
the action $w \mapsto R_w$ of~$F_2$ on itself; 
it is the set $F_2 \times F_2$ equipped with the product
\begin{equation*}
(w_1,u)\, (w_2,v) = (w_1 \, R_u(w_2), uv)
\end{equation*}
for all $w_1, w_2 \in F_2$ and $u, v \in F_2$.
Now consider the map
$$\widehat{\Pal} = (\Pal, \id): F_2 \to F_2 \rtimes F_2 \, ; \ u \mapsto \bigl( \Pal(u),u \bigr) \, .$$
Since $\Pal$ satisfies the cocycle equation~\eqref{Justin},
it follows from a direct computation or from~\eqref{homomorphism} 
that $\widehat{\Pal}$ is a group homomorphism, i.e.,
\begin{equation*}
\widehat{\Pal}(uv) = \widehat{\Pal}(u) \, \widehat{\Pal}(v) 
\end{equation*}
for all $u,v \in F_2$.
We have even better: $\Pal$ being a trivial cocycle, it follows from~\eqref{def-Pal}
and~\eqref{conjugate-trivial} that for all $u\in F_2$,
\begin{equation*}
\widehat{\Pal}(u) = (ab,1)^{-1} \, (1,u) \, (ab,1)
\end{equation*}
in the semi-direct product $F_2 \rtimes F_2$.

\section{Properties}\label{pal-pal}

In this section we show that each element $\Pal(w)$ is a palindrome in~$F_2$
and that our map~$\Pal$ extends de Luca's right iterated palindromic closure;
we also compute the image of~$\Pal(w)$ in the free abelian group~$\ZZ^2$.

By definition, a \emph{palindrome} in~$F_2$ is an element
fixed by the anti-auto\-mor\-phism $\omega$ of~$F_2$ 
introduced in the proof of Proposition~\ref{char-commutator}.
The first property of~$\Pal$ that we prove in this section is the following.

\begin{Prop}\label{palindrome}
We have $\omega \bigl( \Pal(w) \bigr) = \Pal(w)$ for all $w \in F_2$.
\end{Prop}

Let $L_a$ and $L_b$ be the automorphisms defined by~\eqref{values-L}.
One easily checks that for all $w\in F_2$,
\begin{equation}\label{conjLR}
L_a(w) = a \, R_a(w) \, a^{-1}
\quad\text{and}\quad
L_b(w) = b \, R_b(w) \, b^{-1} \, .
\end{equation}

Let $w\mapsto L_w$ be the homomorphism $F_2 \to \Aut(F_2)$ sending
$a$ to~$L_a$ and $b$ to~$L_b$.
We have the following lemma reminiscent of~\eqref{def2-Pal}.

\begin{Lem}\label{coJustin}
We have 
$L_w(ba) = \Pal(w)\,  ba$
for all $w \in F_2$.
\end{Lem}

\pf
We prove the lemma by induction on the length of a reduced word
representing~$w\in F_2$.
If $w=1$, then the lemma holds trivially. 
Suppose that $L_w(ba) = \Pal(w) \, ba$ for some $w\in F_2$ and let
us prove the lemma for $a^{\pm 1}w$ and~$b^{\pm 1}w$.
We have
\begin{eqnarray*}
L_{a^{\pm 1}w}(ba) & = & L_{a^{\pm 1}} \bigl( L_w(ba) \bigr)
= L_{a^{\pm 1}} \bigl( \Pal(w) \, ba \bigr) \\
& = & L_{a^{\pm 1}}(\Pal(w)) \, L_{a^{\pm 1}}(ba)
=  L_{a^{\pm 1}}(\Pal(w)) \, a^{\pm 1}ba \\
& = & a^{\pm 1} \, R_{a^{\pm 1}} (\Pal(w) ) \, a^{\mp 1} a^{\pm 1} ba 
=  \Pal(a^{\pm 1}) \, R_{a^{\pm 1}} (\Pal(w) ) \, ba \\
& = & \Pal(a^{\pm 1}w) \, ba\, .
\end{eqnarray*}
The second equality holds by induction, the fourth by~\eqref{values-L},
the fifth by~\eqref{conjLR}, the sixth by~\eqref{values0},
and the seventh by~\eqref{Justin}.
One proves the lemma for~$b^{\pm 1}w$ in a similar fashion.
\epf

\begin{proof}[Proof of Proposition~\ref{palindrome}]
By~\eqref{def-Pal}, \eqref{def-L}, and Lemma~\ref{coJustin},
\begin{eqnarray*}
ab \, \omega\bigl(\Pal(w) \bigr)
& = & \omega\bigl(\Pal(w) \, ba \bigr)\\
& = & \omega\bigl(L_w(ba) \bigr)
= R_w \bigl( \omega(ba) \bigr) \\
& = & R_w(ab) = ab \Pal(w) 
\end{eqnarray*}
for all $w\in F_2$. The conclusion follows immediately.
\end{proof}

We next show that
the map $\Pal$ satisfies an equation established by Justin~\cite{Ju}
for de Luca's map $\PP : \{a,b\}^* \to \{a,b\}^*$.

\begin{Prop}\label{Pal-Justin-bis}
For all $u$, $v \in F_2$, we have
\begin{equation}\label{Justin-bis}
\Pal(uv) = L_u\bigl( \Pal(v) \bigr)  \Pal(u) \, .
\end{equation}
\end{Prop}

\pf
By Proposition~\ref{palindrome}, $\Pal(u)$, $\Pal(v)$, and $\Pal(uv)$
are palindromes. We thus obtain
\begin{eqnarray*}
L_u\bigl( \Pal(v) \bigr) \Pal(u)
& = & \omega\bigl( R_u \bigl( \omega(\Pal(v)) \bigr) \bigr) \, \omega \bigl( \Pal(u) \bigr) \\
& = & \omega\bigl( \Pal(u) \, R_u(\Pal(v)) \bigr) \\
& = & \omega\bigl( \Pal(uv)\bigr) \\
& = & \Pal(uv) \, .
\end{eqnarray*}
The first equality follows from~\eqref{def-L}
and the third one from~\eqref{Justin}.
\epf

Since $L_a$ and $L_b$ obviously preserve the (free) submonoid~$\{a,b\}^*$ of~$F_2$,
it follows from~\eqref{Justin-bis} that
the restriction of~$\Pal$ to~$\{a,b\}^*$ takes its values in~$\{a,b\}^*$.

\begin{Cor}\label{Cor}
The restriction of~$\Pal$ to the monoid $\{a,b\}^*$ coincides with de Luca's
right iterated palindromic closure $\PP : \{a,b\}^* \to \{a,b\}^*$.
\end{Cor}

\pf
De Luca's map $\PP$ is by~\cite{Ju} the unique map $\{a,b\}^*  \to  \{a,b\}^*$
fixing $a, b$ and satisfying~\eqref{Justin-bis}.
\epf

\begin{Rem}
Our palindromization map $\Pal$ cannot be obtained 
as a right iterated palindromic closure as is the case when $w\in \{a,b\}^*$. 
Indeed, the right iterated palindromic closure in the free monoid~$\{a,b,a^{-1}, b^{-1}\}^*$
of the word~$ab^{-1}a$ is
$$\bigl( (a^+ b^{-1})^+ a \bigr)^+ = (ab^{-1}aa)^+ = ab^{-1}aab^{-1}a\, , $$
whereas $\Pal(ab^{-1}a) = b^{-2}$ by~\eqref{values3}.
Moreover, since $b^+ = b$ and
$$\bigl( (ba)^+ a^{-1}\bigr)^+ = (baba^{-1})^+ = baba^{-1}bab \neq b\, ,$$
we see that an iterated palindromic closure cannot be defined on~$F_2$ in 
a naive way.
\end{Rem}

In~\cite[Sect.~3]{BLR}, the image $\pi \bigl( \PP(w) \bigr)$ in~$\ZZ^2$
was computed for $w\in \{a,b\}^*$. This computation extends easily
to the whole group~$F_2$ as an immediate consequence of the definition of~$\Pal$.

\begin{Prop}\label{M}
For all $w \in F_2$, we have
$$\pi\bigl( \Pal(w) \bigr) = \bigl(M_w - I_2 \bigr) 
\begin{pmatrix}
1 \\
1
\end{pmatrix} 
\, . $$
\end{Prop}

\pf
Since $\pi \circ R_w = M_w \circ \pi $, we have
\begin{eqnarray*}
\pi\bigl( \Pal(w) \bigr) 
& = & \pi\bigl((ab)^{-1} R_w(ab) \bigr) 
= \pi\bigl((ab)^{-1}\bigr) + \pi\bigl( R_w(ab) \bigr) \\
& = & \begin{pmatrix}
-1 \\
-1
\end{pmatrix} 
+ M_w 
\begin{pmatrix}
1 \\
1
\end{pmatrix} 
= \bigl(M_w - I_2 \bigr) 
\begin{pmatrix}
1 \\
1
\end{pmatrix} 
\, . 
\end{eqnarray*}
\epf

\begin{Rem}
By~\eqref{def-Pal} and~\eqref{def2-Pal}, $X = ab$ and $X=ba$
are solutions of the equations
\begin{equation}\label{X}
\Pal(w) = X^{-1} \, R_w(X) \qquad (w\in F_2) \, .
\end{equation}
Using Proposition~\ref{M}, it is easy to check that any solution~$X$ of~\eqref{X}
necessarily satisfies $\pi(X) = (1,1) \in \ZZ^2$. 
Therefore, $X = ab$ and $X=ba$ are the only solutions of~\eqref{X}
in the monoid~$\{a,b\}^*$. Indeed, $\pi(X) = (1,1)$ for $X \in \{a,b\}^*$
means that the word~$X$ has exactly 
one occurrence of~$a$ and one occurrence of~$b$.
In the free group~$F_2$, Equation~\eqref{X} has infinitely many solutions
such as $X_r = (aba^{-1}b^{-1})^r ab$, where $r\in \ZZ$.
Note that $X_0 = ab$ and $X_{-1} = ba$.
\end{Rem}

\section{Elements with equal palindromization}\label{triv-pal}

The aim of this section is to characterize all pairs $(u,v)$ of elements of~$F_2$ 
such that $\Pal(u) = \Pal(v)$. 
We start with the following observation.

\begin{Prop}\label{Kernel2}
We have $\Pal(u) = \Pal(v)$ if and only if~$u^{-1}v \in \Pal^{-1}(1)$.
\end{Prop}

\pf
By~\eqref{def-Pal}, $\Pal(u) = \Pal(v)$ if and only if $R_u(ab) = R_v(ab)$,
which is equivalent to $R_{u^{-1}v}(ab) = ab$, hence to
$\Pal(u^{-1}v) = 1$.
\epf

We are thus reduced to describing the subgroup~$\Pal^{-1}(1)$ of~$F_2$.
To this end
we consider the injective homomorphism $i: B_3 \to \Aut(F_2)$ of 
Proposition~\ref{B3}
and the homomorphism $\beta : F_2 \to B_3$ defined by
\begin{equation*}
\beta(a)  = \sigma_1 \quad\text{and}\quad
\beta(b)  = \sigma_2^{-1} \, .
\end{equation*}

It follows from the definitions that
$R_w = i\bigl( \beta(w) \bigr)$
for all $w\in F_2$.
Therefore, if $\beta(u) = \beta(v)$, then $R_u = R_v$, 
hence $\Pal(u) = \Pal(v)$ by~\eqref{def-Pal}.
In other words, $\Pal(w)$
depends only on the image of~$w$ in~$B_3$.

Let $N$ be the subgroup of $B_3$ generated 
by~$\sigma_1\sigma_2^{-1}\sigma_1^{-1}$. 
Since $B_3$ is torsion-free, $N$ is infinite cyclic.
We now give the promised description.

\begin{Prop}\label{Kernel}
We have $\Pal^{-1}(1) = \beta^{-1}(N)$.
\end{Prop}

\pf
By Corollary~\ref{trivialPal} it is enough to check that 
$$N_0 = \{\, \beta(w) \in B_3 \; |\, R_w(ab) = ab \, \} = \beta\bigl( \Pal^{-1}(1)\bigr)$$
coincides with~$N$.

First, observe that $R_{aba^{-1}} = R_a R_b R_{a^{-1}}$ sends 
$a$ to $aba$ and $b$ to~$a^{-1}$.
Therefore, $R_{aba^{-1}}(ab) = (aba)a^{-1} = ab$. 
Since $\beta(aba^{-1}) = \sigma_1\sigma_2^{-1}\sigma_1^{-1}$
generates~$N$, the latter is contained in~$N_0$.

Conversely, let $w \in F_2$ be such that $\Pal(w) = 1$.
Then by Corollary~\ref{trivialPal}, $R_w$ fixes $ab$; 
hence, $M_w$ fixes $(1,1) \in \ZZ^2$. 
An easy computation shows that any matrix in~$\SL_2(\ZZ)$
fixing~$(1,1)$ is of the form
$$C_r = 
\begin{pmatrix}
1+r & -r\\
r & 1-r
\end{pmatrix} 
\, ,$$
where $r\in \ZZ$.
Since $r\mapsto C_r$ is a group homomorphism, 
we have $C_r = (C_1)^r$. Now,
$$C_1
= \begin{pmatrix}
2 & -1\\
1 & 0
\end{pmatrix} 
= 
\begin{pmatrix}
1 & 1\\
0 & 1
\end{pmatrix} 
\begin{pmatrix}
1 & 0\\
1 & 1
\end{pmatrix} 
\begin{pmatrix}
1 & -1\\
0 & 1
\end{pmatrix} 
= M_{aba^{-1}} \, .
$$
Therefore, $M_w = M_{(aba^{-1})^r}$.
It is well known that $\sigma_1 \mapsto M_a$ and
$\sigma_2 \mapsto M_b^{-1}$ defines a surjective group homomorphism
$B_3 \to \SL_2(\ZZ)$
whose kernel is the infinite cyclic group generated by~$(\sigma_1\sigma_2\sigma_1)^4$.
Since $(\sigma_1\sigma_2\sigma_1)^4$ is central in~$B_3$, we conclude that
$$\beta(w) = (\sigma_1\sigma_2\sigma_1)^{4p} \, (\sigma_1\sigma_2^{-1}\sigma_1^{-1})^r $$
for some~$p\in \ZZ$.
Consequently,
$$R_w = i \bigl( \beta(w)\bigr) 
= (R_{(ab^{-1}a)^4})^{p} \, (R_{aba^{-1}})^r \, .$$
Since $R_{aba^{-1}}$ fixes~$ab$ and $R_w$ is assumed to fix~$ab$, we obtain
$$ab = R_w(ab) = (R_{(ab^{-1}a)^4})^{p}(ab)\, .$$
By~\eqref{R-center30}, $R_{(ab^{-1}a)^4}$ is the conjugation by~$b a b^{-1}a^{-1}$;
hence, $(R_{(ab^{-1}a)^4})^p$ is the conjugation by~$(b a b^{-1}a^{-1})^p$.
We thus obtain
$$ab = (b a b^{-1}a^{-1})^p \, ab \, (b a b^{-1}a^{-1})^{-p} \, .$$
If $p<0$, then the right-hand side is a reduced word different from~$ab$.
If $p>0$, then the right-hand side represents the same element of~$F_2$
as the reduced word $(b a b^{-1}a^{-1})^{p-1} \, ba \, (b a b^{-1}a^{-1})^{-p}$,
which is also different from~$ab$. It follows that necessarily $p=0$
and hence
$\beta(w) = (\sigma_1\sigma_2^{-1}\sigma_1^{-1})^r $. 
We have thus proved that $N_0 \subset N$.
\epf

Observe that $\Pal^{-1}(1) = \beta^{-1}(N)$ is the product (in~$F_2$)
of the cyclic group generated by~$aba^{-1}$ by 
the kernel $\Ker(\beta)$ of $\beta : F_2 \to B_3$. 
In view of the presentation~\eqref{presentationB3}, 
$\Ker(\beta)$ is the normal subgroup of~$F_2$
generated by~$ab^{-1}aba^{-1}b$ (or equivalently by $(aba^{-1})(bab^{-1})$).

By~\cite{Lu}, if $u$, $v \in \{a,b\}^*$, then $\Pal(u) = \Pal(v)$ if and only if~$u=v$.
This sharply contrasts with the previous results.

\section{The profinite topology}\label{profinite}

The \emph{profinite topology} on~$F_2$ is the coarsest topology such that every group homomorphism 
from~$F_2$ into a finite group is continuous. 
This topology was introduced by Marshall Hall in~\cite{Ha} and is sometimes called the Hall topology;
see~\cite{PR} and~\cite{Re} for applications to automata theory and semigroup theory.

\begin{Thm}\label{continu}
The map $\Pal: F_2 \to F_2$ is continuous for the profinite topology.
\end{Thm}

Thus the map $\Pal$ yields an example of a non-trivial combinatorially-defined
continuous function on~$F_2$. 
For other examples of such fonctions, see the subword functions in~\cite{MR}.

\pf
Since $\Pal(w)  = b^{-1} a ^{-1}\, R_w(a) \, R_w(b)$, it suffices to prove that the map
$F_2 \to F_2 \times F_2\, ; \; w \mapsto (R_w(a), R_w(b))$ is continuous.
The latter is equivalent to the continuity of 
$$w \mapsto \bigl((\varphi \circ R_w)(a), (\varphi \circ R_w)(b) \bigr)$$
for all group homomorphisms $\varphi : F_2 \to G$ into a finite group~$G$
(equipped with the discrete topology).
It thus suffices to check that for each $(g,h) \in G\times G$, 
the set $X(g,h)$ of~$w\in F_2$ such that 
$$((\varphi \circ R_w)(a), (\varphi \circ R_w)(b) ) = (g,h)$$
is a union of cosets of subgroups of finite index of~$F_2$.

Since $F_2$ acts on the left on itself by $w \mapsto R_w$, it acts
on the right on the set $\Hom(F_2,G)$ of group homomorphisms of~$F_2$ into~$G$
by 
$\varphi \cdot w = \varphi \circ R_w$
for all $\varphi \in \Hom(F_2,G)$ and $w\in F_2$.
Now, $\Hom(F_2,G)$ is in bijection with $G \times G$ via the map
$\varphi \mapsto (\varphi(a), \varphi(b))$.
Hence, $F_2$ acts on the right on $G \times G$ in such a way that
$$\bigl(\varphi(a), \varphi(b) \bigr) \cdot w = \bigl((\varphi \circ R_w)(a), (\varphi \circ R_w)(b) \bigr)$$
for all $\varphi \in \Hom(F_2,G)$ and $w\in F_2$.
Therefore, the above-defined set~$X(g,h)$ coincides with the set
of $w\in F_2$ such that $(\varphi(a), \varphi(b)) \cdot w = (g,h)$.
This set is a coset in~$F_2$
of the stabilizer $G(g,h) = \{w \in F_2 \, |\, (g,h) \cdot w = (g,h)\}$.
We conclude by observing that each stabilizer of an action of a group~$F$ 
on a \emph{finite} set is necessarily a finite index subgroup of~$F$.
\epf

\begin{Cor}\label{unique}
The map $\Pal: F_2 \to F_2$ is the unique continuous extension to~$F_2$ 
of de Luca's right iterated palindromic closure~$\PP : \{a,b\}^* \to \{a,b\}^*$.
\end{Cor}

\pf
It is well known that the submonoid $\{a,b\}^*$ is dense in~$F_2$ for the
profinite topology. Indeed, since  $(w^{n!})_n$ converges to~$1$,
the sequence $(w^{n!-1})_n$ converges to~$w^{-1}$. 
Applying this remark to $w = a,b$,
we see that $a^{-1}$ and $b^{-1}$ are limits of sequences of elements of~$\{a,b\}^*$.
Therefore the map $\PP : \{a,b\}^* \to \{a,b\}^*$ has at most one continuous
extension to~$F_2$.
We conclude with Corollary~\ref{Cor} and Theorem~\ref{continu}.
\epf

\begin{Rem}
Though $\Pal$ is continuous for the profinite topology, it is
\emph{not} continuous for the pro-$p$-finite topology on~$F_2$,
where $p$ is a prime number.
Recall that the \emph{pro-$p$-finite topology} is the coarsest topology 
such that every group homomorphism 
from~$F_2$ into a finite $p$-group is continuous. 
By definition of the pro-$p$-finite topology, the sequence $(w^{p^n})_n$ converges to~$1$
for any $w\in F_2$.
We claim that if $w =  a^{-1} b$,
then the sequence $\Pal(w^{p^n})$ does not converge to~$\Pal(1)= 1$.
It is enough to check that that the vector-valued sequence $\pi( \Pal(w^{p^n}) )\in \ZZ^2$ 
does not converge to~$\pi(1)$, which is the zero vector.

The matrix $M_w\in \SL_2(\ZZ)$ corresponding to~$w=  a^{-1} b$ is equal to
$$M_w = M_a^{-1} M_b =
\begin{pmatrix}
0 & -1\\
1 & 1
\end{pmatrix}\, ,$$
which is of order six. 
So $M_w^k$ takes only six values when $k$ runs over the positive integers. 
One checks that if $k$ is not divisible by~$6$, then
$$\bigl(M_w^k - I_2 \bigr)
\begin{pmatrix}
1 \\
1
\end{pmatrix} 
\neq \begin{pmatrix}
0 \\
0
\end{pmatrix} 
\, . $$
Since no power $p^n$ of a prime number is divisible by~$6$, by the previous
observation and by Proposition~\ref{M},
$$\pi\bigl( \Pal(w^{p^n}) \bigr) 
= \bigl(M_w^{p^n} - I_2 \bigr) 
\begin{pmatrix}
1 \\
1
\end{pmatrix} 
\neq
\begin{pmatrix}
0 \\
0
\end{pmatrix} 
\, . $$
This shows that the sequence $\pi( \Pal(w^{p^n}))$ does not converge to~$\pi(1)$,
and hence $\Pal$ is not continuous for the pro-$p$-finite topology.
\end{Rem}

\begin{Rem}
As a consequence of Theorem~\ref{continu}, 
de Luca's right iterated palindromic closure $\PP :\{a,b\}^* \to \{a,b\}^*$ 
is continuous for the restriction of the profinite topology to~$\{a,b\}^*$. 
Nevertheless, $w\mapsto w^+$ is not continuous.
Indeed, we have seen in the proof of Corollary~\ref{unique}
that for any $w\in F_2$,
the sequence $(w^{n!})_n$ converges to~$1$ in the profinite topology.
In particular, the sequence $(ab^{n!})_n$ converges to~$a$. 
Now, $(ab^{n!})^+ = ab^{n!}a$. Therefore,
the sequence $(ab^{n!})^+$ converges to~$aa$, 
which is different from~$a$.
\end{Rem}

\section{A characterization of the image of~$\Pal$}\label{sect-image}

The aim of this section is to characterize the elements of~$F_2$
belonging to the image $\Pal(F_2)$ of~$\Pal$.
The notation $g \sim h$ used in the sequel means that $g$ and $h$ are conjugate elements 
of~$F_2$.

By \eqref{def-Pal} and~\eqref{def2-Pal}, for any $w\in F_2$,
\begin{eqnarray*}
ab \, \Pal(w) & = & R_w(ab) 
= R_w(a) \, R_w(ba) \, R_w(a)^{-1} \\
& = & R_w(a) \bigl(ba\, \Pal(w) \bigr) R_w(a)^{-1} \, .
\end{eqnarray*}
Thus, $abg \sim bag$ for all $g\in \Pal(F_2)$.
In~\cite{Pi}, Pirillo proved that, if $A$ is an alphabet and $g$ is a word in~$A$
such that $abg \sim bag$ for two distinct letters $a,b \in A$, 
then $g$ is a central word in the alphabet $\{a,b\}$;
hence by de~Luca~\cite{Lu}, $g = \Pal(w)$ for some~$w\in  \{a,b\}^*$.
We have the following extension of Pirillo's result to our palindromization map.

\begin{Thm}\label{image}
An element $g\in F_2$ belongs to $\Pal(F_2)$ if and only if
$abg$ and $ba g$ are conjugate in~$F_2$.
\end{Thm}

\pf
We assume that $g$ is a reduced word in~$F_2$ such that $abg \sim bag$.
We shall prove by induction on the length of~$g$ that $g = \Pal(w)$
for some $w\in F_2$.

If $g$ is of length zero, then $g = 1 = \Pal(1)$. 
Now, suppose that $g$ is of length~$>0$.
If there is no cyclic cancellation in~$abg$, then there is none in~$bag$ since $abg \sim bag$;
in this case,  $abg$ and $bag$ are conjugate in the free monoid $\{a,b, a^{-1}, b^{-1}\}^*$. 
Applying the above-mentioned theorem by Pirillo,
we conclude that $g = \Pal(w)$ for some~$w\in  \{a,b\}^*$.

If the reduced word $g$ starts by~$b^{-1}$, write $g = b^{-1}g'$, where
$g'$ is a reduced word. 
We claim that $g'$ ends with~$b^{-1}$.
Indeed, 
$$ag' = abg \sim bag = bab^{-1}g'\, .$$
If $g'$ did not end with~$b^{-1}$, then $bab^{-1}g'$ would be cyclically reduced
and conjugate to~$ag'$, which is shorter than~$bab^{-1}g'$;
this is impossible and thus establishes the claim.
It follows that $g = b^{-1} h b^{-1}$ for some shorter reduced word~$h$.
Let us show that
\begin{equation}\label{sim}
ab^{-1} h \sim b^{-1} ah \, .
\end{equation}
Indeed, $ab^{-1} h =  agb \sim bag \sim  abg 
= ahb^{-1} \sim b^{-1}ah$.
To~\eqref{sim} we apply
the inverse~$\tau^{-1}$ of the automorphism~$\tau$ of~$F_2$ 
introduced in Section~\ref{automorphisms}, thus obtaining
$$ba \, \tau^{-1}(h)  = \tau^{-1}(ab^{-1} h) 
\sim \tau^{-1}(b^{-1} ah) =  ab \, \tau^{-1}(h)\, .$$
Since $\tau^{-1}(h)$ is a reduced word of length less than the length of~$g$,
we may apply the induction hypothesis to~$\tau^{-1}(h)$ and deduce that
$\tau^{-1}(h) = \Pal(u)$ for some~$u\in F_2$.
Let us compute $\Pal(ab^{-1}au)$. 
Using~\eqref{Justin}, \eqref{values3}, and~\eqref{R-center10}
successively, we obtain
\begin{eqnarray*}
\Pal(ab^{-1}au) & = & \Pal(ab^{-1}a)\, R_{ab^{-1}a}\bigl( \Pal(u) \bigr) \\
& = & b^{-2} \, R_{ab^{-1}a}\bigl( \tau^{-1}(h) \bigr) \\
& = & b^{-2} \, (bhb^{-1}) = g \, .
\end{eqnarray*}
The case where $g$ starts by~$a^{-1}$ is treated in a similar manner.
\epf

\begin{Cor}\label{closedimage}
The subset $\Pal(F_2)$ is closed in~$F_2$ for the profinite topology.
\end{Cor}

\pf
By Theorem~\ref{image}, it is enough to check 
that the subset of elements $g\in F_2$
such that $abg \sim bag$ is closed in~$F_2$. 
It suffices to consider two sequences $(u_n)$ and
$(v_n)$ converging in~$F_2$ respectively to~$u$ and~$v$ and such that $u_n \sim v_n$
for all~$n$
and to show that $u$ and $v$ are conjugate in~$F_2$.
In this situation there are elements $x_n\in F_2$ such that $u_n x_n = x_n v_n$.
Since the profinite completion of~$F_2$ is compact, there is a subsequence of~$(x_n)$
that converges to an element~$x$ of the completion such that $ux = xv$.
Therefore, $u$ and $v$ are conjugate in the completion, hence in all finite quotients of~$F_2$.
By~\cite[Prop.~I.4.8]{LS}, $u$ and $v$ are conjugate in~$F_2$.
\epf

\begin{Qu}
De Luca~\cite{Lu} gives a simple algorithm to recover~$w \in \{a,b\}^*$ 
from $g = \PP(w)$, namely $w$ is the sequence of the letters of~$g$ that immediately follow
all palindromic prefixes (including the empty word) of~$g$.
It is possible to extract from the proof of Theorem~\ref{image} an algorithm 
producing $w \in F_2$ out of $g = \Pal(w)$; 
the element~$w$ obtained in this way is not necessarily of shortest length.
We raise the question of finding a simple constructive procedure that produces 
a (unique?)\ element of shortest length out of its image under~$\Pal$.
\end{Qu}

\appendix

\section{Serre's non-abelian cohomology}\label{Serre}

Let $G$ be a group acting on another group~$E$
via a group homomorphism $G \to \Aut(E)\, ; \ u \mapsto R_u$
(the group~$E$ is not assumed to be abelian).
Following~\cite[Chap.~1, Sect.~5.1]{Se}, 
we call \emph{cocycle of $G$ with values in~$E$}
any map $\varphi : G \to E$ verifying the \emph{cocycle condition}
\begin{equation}\label{cocycle}
\varphi(uv) = \varphi(u) \, R_u (\varphi(v))
\end{equation}
for all $u,v \in G$.

Two cocycles $\varphi, \varphi': G \to E$ are said to be \emph{cohomologous} 
if there exists $X \in E$ such that
\begin{equation}\label{coboundary}
\varphi'(u) = X^{-1} \, \varphi (u) \, R_u(X)
\end{equation}
for all $u \in G$.
A cocycle~$\varphi$ is said to be \emph{trivial} if it is cohomologous to
the cocycle sending each $u\in G$ to the neutral element~$1$ of~$E$; 
it follows from~\eqref{coboundary} that $\varphi$ is trivial if and only if 
there is $X \in E$ such that for all $u\in G$,
\begin{equation}\label{null}
\varphi(u) = X^{-1} R_u(X) \, .
\end{equation} 

Let $E \rtimes G$ be the semi-direct product associated to
the action of~$G$ on~$E$; 
it is the set $E \times G$ equipped with the product
\begin{equation*}
(X,u)\, (Y,v) = (X \, R_u(Y), uv)
\end{equation*}
for all $X,Y \in E$ and $u, v \in G$.
To a map $\varphi : G \to E$ we associate the map
$\widehat{\varphi} = (\varphi, \id): G \to E \rtimes G$.
Then $\varphi$ satisfies the cocycle equation~\eqref{cocycle} if and only
if $\widehat{\varphi}$ is a group homomorphism, i.e.,
\begin{equation}\label{homomorphism}
\widehat{\varphi}(uv) = \widehat{\varphi}(u) \, \widehat{\varphi}(v) \quad (u,v\in G)  \, .
\end{equation}

It is also easy to check that, if $\varphi, \varphi': G\to E$ and $X\in E$ satisfy~\eqref{coboundary}, 
then $\widehat{\varphi}$ and $\widehat{\varphi'}$
are conjugate in the group~$E \rtimes G$; more precisely,
\begin{equation*}
\widehat{\varphi'}(u) = (X,1)^{-1} \, \widehat{\varphi}(u) \, (X,1) \in E \rtimes G
\end{equation*}
for all $u\in G$. In particular, if~$\varphi$ is a trivial cocycle with $X\in E$ as in~\eqref{null},
then for all $u\in G$,
\begin{equation}\label{conjugate-trivial}
\widehat{\varphi}(u) = (X,1)^{-1} \, (1,u) \, (X,1) \, .
\end{equation}

\end{document}